\documentclass[12pt]{article}

\usepackage[margin=1in]{geometry}
\usepackage[usenames]{xcolor}
\usepackage{amssymb}
\usepackage{relsize}
\usepackage{amsmath}
\usepackage{amsthm}
\usepackage{amsfonts}
\usepackage{amscd}
\usepackage{graphicx}
\usepackage{hyperref}
\usepackage{thmtools}
\usepackage{thm-restate}
\usepackage{tikz-cd} 
\usepackage{placeins}

\begin{document}

\theoremstyle{plain}
\newtheorem{theorem}{Theorem}
\newtheorem{corollary}[theorem]{Corollary}
\newtheorem{lemma}[theorem]{Lemma}
\newtheorem{proposition}[theorem]{Proposition}

\theoremstyle{definition}
\newtheorem{definition}[theorem]{Definition}
\newtheorem{example}[theorem]{Example}
\newtheorem{conjecture}[theorem]{Conjecture}
\newtheorem{question}[theorem]{Question}

\theoremstyle{remark}
\newtheorem{remark}[theorem]{Remark}

\begin{center}
\vskip 1cm{\LARGE\bf 
Some Bounds for Number of Solutions to $ax+by+cz=n$ and their Applications \\ 
\vskip .1in

}
\vskip 1cm
\large
Damanvir Singh Binner \\
Department of Mathematics\\
Simon Fraser University \\
Burnaby, BC V5A 1S6\\
Canada \\
 dbinner@sfu.ca
\end{center}

\vskip .2in

\begin{abstract}
In a recent work, the present author developed an efficient method to find the number of solutions of $ax+by+cz=n$ in non-negative integer triples $(x,y,z)$ where $a,b,c$ and $n$ are given natural numbers. In this note, we use that formula to obtain some simple looking bounds for the number of solutions of $ax+by+cz=n$. Using these bounds, we solve some special cases of a problem related to the generalization of Frobenius coin problem in three variables. Moreover, we use these bounds to disprove a recent conjecture of He, Shiue and Venkat regarding the solution structure of $ax+by+cz=n$.
 \end{abstract}
 
 \section{Introduction}
 
 Let $a,b,c$ and $n$ be given natural numbers such that $\gcd(a,b) = \gcd(b,c) = \gcd(c,a) =1$. We recall the formula for the number of solutions $N(a,b,c;n)$ of $ax+by+cz=n$ in non-negative integer triples $(x,y,z)$ described in \cite[Theorem 5]{Binner}. We restate the formula here. For that, we need to introduce some notation.

\begin{itemize}
	\item Define $ b'_1$ such that  $b'_1 \equiv - nb^{-1}$ (mod $a$)  with $ 1\leq b'_1 \leq a$. Moreover, define $ c'_1$ such that $c'_1 \equiv bc^{-1}$ (mod $a$)  with $ 1\leq c'_1 \leq a$.
	\item Define  $ c'_2$ such that  $c'_2 \equiv - nc^{-1}$ (mod $b$)  with $ 1\leq c'_2 \leq b$. Moreover, define $ a'_2$ such that $a'_2 \equiv ca^{-1}$ (mod $b$)  with $ 1\leq a'_2\leq b$.
	\item Define $ a'_3$ such that $a'_3 \equiv - na^{-1}$ (mod $c$)  with $ 1\leq a'_3 \leq c$. Moreover, define  $b'_3$ such that $b'_3 \equiv ab^{-1}$ (mod $c$)  with $ 1\leq b'_3 \leq c$.
	\item Define $N_1 = n(n + a + b +c) + cbb'_1(a+1-c'_1(b'_1-1)) + acc'_2(b+1 - a'_2(c'_2-1))$ $+ baa'_3 (c+1-b'_3(a'_3-1)) $.
\end{itemize}

\begin{theorem}(B.(2020))
\label{MainThm}
Let $a$, $b$, $c$, and $n$ be given positive integers such that $\gcd(a,b) = \gcd(b,c) = \gcd(c,a) =1$. With the notation above, the number of nonnegative integer solutions of the equation $ax + by +cz = n$ is given by 
$$N(a,b,c;n) =   \frac{N_1}{2abc} + \sum_{i=1}^{b'_1-1} \left\lfloor \frac{ic'_1}{a} \right\rfloor + \sum_{i=1}^{c'_2-1} \left\lfloor \frac{ia'_2}{b} \right\rfloor + \sum_{i=1}^{a'_3-1} \left\lfloor \frac{ib'_3}{c} \right\rfloor  - 2 .$$
\end{theorem}

\section{Bounds for $N(a,b,c;n)$}

  First, we rewrite the expression for $N(a,b,c;n)$ in another convenient form, from which it is easy to deduce some nice bounds for $N(a,b,c;n)$. Recall the definition of the sawtooth function
 
 $$ 
((x)) = 
\begin{cases}
\{x\}-\frac{1}{2}, & \text{if } x \notin \mathbb{Z} \\
0, & \text{if }x \in \mathbb{Z},
\end{cases}
$$

where $\{x\}$ denotes the fractional part of $x$. Note that $-\frac{1}{2} < ((x)) < \frac{1}{2}$ for any $x$. Substituting the value of $N_1$ in the expression for $N(a,b,c;n)$ given in Theorem \ref{MainThm} and simplifying further, we can rewrite $N(a,b,c;n)$ as follows.

\begin{equation}  \label{AnoExp}
\begin{split}
 N(a,b,c;n) &= \frac{n(n+a+b+c)}{2abc} - \frac{1}{2} + \left( \frac{b'_1}{2a} + \frac{c'_2}{2b} + \frac{a'_3}{2c} \right) \\
 &- \sum_{i=1}^{b'_1-1} \left( \left( \frac{ic'_1}{a} \right) \right) - \sum_{i=1}^{c'_2-1} \left( \left( \frac{ia'_2}{b} \right) \right) - \sum_{i=1}^{a'_3-1} \left( \left( \frac{ib'_3}{c} \right) \right).  
 \end{split}
  \end{equation}

 To verify that this expression is equivalent to the one in Theorem \ref{MainThm}, just substitute the sawtooth function $((x))$ with $x-\left \lfloor x \right \rfloor-\frac{1}{2}$ in all the three sums, since the fractions  $\frac{ic'_1}{a}$,  $\frac{ia'_2}{b}$ and  $\frac{ib'_3}{c}$ are never integers for the given values of $i$ in each of the three sums. Thus, substituting the sawtooth functions in terms of floor functions in the summations in \eqref{AnoExp}, and then simplifying the sums gives us the expression for $N(a,b,c;n)$ given in Theorem \ref{MainThm}. Next, we use this expression in \eqref{AnoExp} to obtain some useful bounds for $N(a,b,c;n)$.
 
 \begin{theorem}
 \label{Bounds}
 Let $a$, $b$, $c$, and $n$ be given positive integers such that $\gcd(a,b) = \gcd(b,c) = \gcd(c,a) =1$.  Further let $N(a,b,c;n)$ denote the number of nonnegative integer solutions of the equation $ax + by +cz = n$. Then, $$ \frac{n(n+a+b+c)}{2abc} - \frac{a+b+c}{2} < N(a,b,c;n) < \frac{n(n+a+b+c)}{2abc} + \frac{a+b+c}{2}. $$
 \end{theorem}
 
 \begin{proof}
 
 Using the expression for $N(a,b,c;n)$ in \eqref{AnoExp}, we have
 
 \begin{align*}
 \left|N(a,b,c;n) - \frac{n(n+a+b+c)}{2abc}\right| &\leq \left|- \frac{1}{2} + \left( \frac{b'_1}{2a} + \frac{c'_2}{2b} + \frac{a'_3}{2c} \right)\right| \\
 &+ \left|\sum_{i=1}^{b'_1-1} \left( \left( \frac{ic'_1}{a} \right) \right)\right| + \left|\sum_{i=1}^{c'_2-1} \left( \left( \frac{ia'_2}{b} \right) \right)\right| + \left|\sum_{i=1}^{a'_3-1} \left( \left( \frac{ib'_3}{c} \right) \right)\right| \\
 &\leq \frac{1}{2} + \frac{a-1}{2} + \frac{b-1}{2} + \frac{c-1}{2} \\
 &< \frac{a+b+c}{2}.
\end{align*}

This completes the proof of Theorem \ref{Bounds}.

\end{proof}

\section{An application to $R_k(a,b,c)$}

For brevity of notation, we set $\alpha = \frac{a+b+c}{2}$ and $\beta = 2abc$. Thus, the bounds in Theorem \ref{Bounds} can be rewritten as 
\begin{equation}
\label{Rebound}
 \frac{n(n+2\alpha)}{\beta} - \alpha < N(a,b,c;n) < \frac{n(n+2\alpha)}{\beta} + \alpha. 
 \end{equation}

These bounds help us to solve a special case of a problem discussed by Bardomero and Beck in \cite{BB} and studied further in \cite{Woods}. We discuss the problems here in the context of three variables, though they are defined for any number of variables. Let $a,b$ and $c$ be given positive integers such that $\gcd(a,b,c)=1$. Let $R_k(a,b,c)$ consists of all integers $n$ such that the equation $ax+by+cz=n$ have exactly $k$ solutions. Then, Bardomero and Beck \cite{BB} suggested the questions of finding the largest number $g_k(a,b,c)$ in $R_k(a,b,c)$ and the cardinality of $|R_k(a,b,c)|$. In general, let $f(t)$ denotes the number of solutions of $ax+by+cz=t$. Then, Woods \cite{Woods} defined the following quantities. 

\begin{itemize}
\item $g_{=k}(a,b,c)$ is the maximum $t$ such that $f(t) = k$. 
\item $h_{=k}(a,b,c)$ is the minimum $t$ such that $f(t) = k$.  
\item $c_{=k}(a,b,c)$ is the number of $t$ such that $f(t) = k$.
\item $s_{=k}(a,b,c)$ is the sum of $t$ such that $f(t) = k$.
\end{itemize} 
 
 In terms of $R_k(a,b,c)$, we can rewrite these as follows. 
 \begin{align*}
 g_{=k}(a,b,c) = \max\{t: t \in R_k(a,b,c)\}, \\
 h_{=k}(a,b,c) = \min\{t: t \in R_k(a,b,c)\}, \\
c_{=k}(a,b,c) = |\{t: t \in R_k(a,b,c)\}|, \\
s_{=k}(a,b,c) = \sum\{t: t \in R_k(a,b,c)\}.
 \end{align*}
 
  Using our bounds for $N(a,b,c;n)$, we obtain expressions for these numbers when $k$ is sufficiently large. We define the following notation. 
  
  \subsection{An algorithm to find $R_k(a,b,c)$ if $a,b$ and $c$ are pairwise coprime}
  
  First, we focus on the case $\gcd(a,b) = \gcd(b,c) = \gcd(c,a) =1$. Once we are done with this case, we will study $R_k(a,b,c)$ for any natural numbers $a$, $b$ and $c$.
  
  \begin{itemize}
  \item Recall that $\alpha = \frac{a+b+c}{2}$ and $\beta = 2abc$. Then, $$ M = \left \lfloor \frac{(2\alpha\beta-1)^2 - 4\alpha^2}{4\beta} + \alpha \right \rfloor + 1, $$
  \item  For any natural number $k$, define $$\gamma_k = \left \lfloor \sqrt{\beta(k+\alpha)+\alpha^2} - \alpha \right \rfloor$$ and $$\delta_k = \left \lfloor \sqrt{\beta(k-\alpha)+\alpha^2} - \alpha \right \rfloor. $$ 
  \end{itemize}
  
  \begin{theorem}
  \label{R_k}
   Let $a$, $b$ and $c$ be given positive integers such that $\gcd(a,b) = \gcd(b,c) = \gcd(c,a) =1$. Suppose $k \geq M$. If $\gamma_k = \delta_k$, then $R_k(a,b,c) = \emptyset$. Otherwise, suppose $\gamma_k \neq \delta_k$. Then, $R_k(a,b,c) = \{\gamma_k\}$ or $R_k(a,b,c) = \emptyset$, depending on whether $$ ax+by+cz = \gamma_k $$ has exactly $k$ solutions or not.
  \end{theorem}
  
  \begin{proof}
  
  Suppose $f(t)=k$. Then, by the above bounds in \eqref{Rebound}, we get that $$ k- \alpha < \frac{t(t+2\alpha)}{\beta} < k+ \alpha.$$ Equivalently, $$ \beta(k-\alpha)+\alpha^2 < (t+\alpha)^2 < \beta(k+\alpha)+\alpha^2. $$ That is, 
  \begin{equation}
  \label{gammadelta}
   \sqrt{\beta(k-\alpha)+\alpha^2} - \alpha < t < \sqrt{\beta(k+\alpha)+\alpha^2} - \alpha. 
   \end{equation}
    Thus, if $\gamma_k = \delta_k$, then there is no possible value of $t$ and $R_k(a,b,c) = \emptyset$. 
 
  Next suppose $\gamma_k \neq \delta_k$. We show that if $k \geq M$, then the left hand side and the right hand side of \eqref{gammadelta} differ by less than $1$, and thus $t$ can be determined from this inequality. Since $k \geq M$, we have $$ k > \frac{(2\alpha\beta-1)^2 - 4\alpha^2}{4\beta} + \alpha.$$ Therefore, $$ 2\alpha\beta-1 <  2\sqrt{\beta(k-\alpha)+\alpha^2}. $$ Thus, $$ (\beta(k+\alpha)+\alpha^2) < (\beta(k-\alpha)+\alpha^2) + 1 + 2\sqrt{\beta(k-\alpha)+\alpha^2}. $$  That is, 
 
  \begin{equation}
  \label{Help}
   \sqrt{\beta(k+\alpha)+\alpha^2}  <  \sqrt{\beta(k-\alpha)+\alpha^2} + 1. 
   \end{equation}
   
   From \eqref{gammadelta} and \eqref{Help}, it follows that there can be at most one possible value of $t$ and that is $\gamma_k$. This completes the proof of the theorem.
     
  \end{proof}
 
 Let $a$, $b$ and $c$ be given positive integers such that $\gcd(a,b) = \gcd(b,c) = \gcd(c,a) =1$. We say that a number $k$ is of \emph{category I with respect to $a,b$ and $c$} if $\gamma_k \neq \delta_k$ and the equation $$ ax+by+cz = \gamma_k$$ has exactly $k$ solutions. Otherwise, we say that $k$ is of \emph{category II with respect to $a,b$ and $c$}. When there is no confusion about $a,b$ and $c$, we just say that $k$ is of category I or $k$ is of category II. 
 
 Thus, if $k \geq M$ is of category I, then by Theorem \ref{R_k}, $R_k(a,b,c) = \left\{\gamma_k \right\}$.  Thus, $c_{=k}(a,b,c) = 1$ and $$ g_{=k}(a,b,c) =  h_{=k}(a,b,c) = s_{=k}(a,b,c) = \gamma_k.$$ Otherwise, if $k \geq M$ is of category II, then by Theorem \ref{R_k}, $R_k(a,b,c) = \emptyset$. Thus, $c_{=k}(a,b,c) = 0$. Also, by the convention that empty sum is $0$, we get $s_{=k}(a,b,c) =0$. However, in this case, $g_{=k}(a,b,c)$ and  $h_{=k}(a,b,c)$ are not defined.
 
 Thus, for any $k \geq M$, the problem of finding $R_k(a,b,c)$, $g_{=k}(a,b,c)$,  $h_{=k}(a,b,c)$, $c_{=k}(a,b,c)$ and $s_{=k}(a,b,c)$ reduces to determining the category of $k$. For that, we need to find the number of solutions of the equation $$ ax+by+cz= \gamma_k. $$ However, we can easily do that using the algorithm described in \cite[Section 2.3]{Binner}.
  
We summarize our algorithm for finding the quantities $R_k(a,b,c)$, $g_{=k}(a,b,c)$, $h_{=k}(a,b,c)$, $c_{=k}(a,b,c)$ and $s_{=k}(a,b,c)$ for given positive numbers $a,b$ and $c$ with $\gcd(a,b) = \gcd(b,c) = \gcd(c,a) =1$, and given $k \geq M$, where $$ M = \left \lfloor \frac{(2\alpha\beta-1)^2 - \alpha^2}{\beta} + \alpha \right \rfloor + 1, $$ where $\alpha = \frac{a+b+c}{2}$ and $\beta = {2abc}$.
 
 \begin{enumerate}
 \item Evaluate the quantities $\gamma_k = \left \lfloor \sqrt{\beta(k+\alpha)+\alpha^2} - \alpha \right \rfloor$ and $\delta_k = \left \lfloor \sqrt{\beta(k-\alpha)+\alpha^2} - \alpha \right \rfloor$. 
 \item Determine the category of $k$. We do that in two steps. If $\gamma_k = \delta_k$, then $k$ is of category II. Otherwise, find the number of solutions of $$ ax+by+cz =  \gamma_k $$ using the algorithm described in \cite[Section 2.3]{Binner}. If the number of solutions equals $k$, then $k$ is of category I. Otherwise $k$ is of category II.
 \item If $k$ is of category I, then  $R_k(a,b,c) = \left\{\gamma_k \right\}$, $c_{=k}(a,b,c) = 1$ and 
  $$ g_{=k}(a,b,c) =  h_{=k}(a,b,c) = s_{=k}(a,b,c) = \gamma_k.$$
  \item If $k$ is of category II, then $R_k(a,b,c) = \emptyset$, and $$c_{=k}(a,b,c) = s_{=k}(a,b,c) =0.$$ However, in this case, $g_{=k}(a,b,c)$ and  $h_{=k}(a,b,c)$ are not defined.
  
 \end{enumerate}
 
 \subsection{An example}
 
 We demonstrate our algorithm for an example. Let $a=37$, $b=23$ and $c=16$. Then $\alpha = 38$, $\beta = 27232$ and $M = 157291918$. We illustrate our example for three values of $k$. First suppose  $k=157295111$. Then, $$ \gamma_k = \delta_k = 2069614. $$ Thus, by the second step of the algorithm, $k$ is of category II, and then by the fourth step, $R_k(37,23,16) = \emptyset$ for $k=157295111$.  Thus, in this case $$c_{=k}(37,23,16) = s_{=k}(37,23,16) = 0. $$ However, in this case, $g_{=k}(37,23,16)$ and  $h_{=k}(37,23,16)$ are not defined.
 
 Next, suppose $k=157295072$. In this case, $\gamma_k = 2069614$ and $\delta_k = 2069613$. Thus, $ \gamma_k \neq \delta_k$. Therefore, we need to find the number of solutions $$ 37x + 23y + 16z = 2069614. $$ Using the formula in Theorem \ref{MainThm}, we find that the number of solutions of $ 37x + 23y + 16z =  2069614$ is given by $$ 157295066 + \sum_{i=1}^{3} \left \lfloor \frac{13i}{37} \right \rfloor + \sum_{i=1}^{3} \left \lfloor \frac{11i}{23} \right \rfloor + \sum_{i=1}^{9} \left \lfloor \frac{3i}{16} \right \rfloor. $$ Using the algorithm in \cite[Section 2.3]{Binner} (Basically reciprocity relation of \cite[Lemma 7]{Binner}), we easily get that 
 
\begin{equation}
 \begin{aligned}
 \sum_{i=1}^{3} \left \lfloor \frac{13i}{37} \right \rfloor &= 1, \\
 \sum_{i=1}^{3} \left \lfloor \frac{11i}{23} \right \rfloor &= 1, \\
\sum_{i=1}^{9} \left \lfloor \frac{3i}{16} \right \rfloor &= 4.
 \end{aligned}
\end{equation}

Thus, the number of solutions of $ 37x + 23y + 16z = 2069614$ is equal to $157295072$. Therefore, $k= 157295072$ is of category I, and we get that in this case $R_k(37,23,16) = \{2069614\}$. Thus, $c_{=k}(37,23,16) = 1$ and 
  $$ g_{=k}(37,23,16) =  h_{=k}(37,23,16) = s_{=k}(37,23,16) = 2069614.$$ 
  
Finally, we consider $k=157294925$. In this case, $\gamma_k = 2069613$ and $\delta_k = 2069612$. Thus, $ \gamma_k \neq \delta_k$. Therefore, we need to find the number of solutions $$ 37x + 23y + 16z = 2069613. $$ Using the formula in Theorem \ref{MainThm}, we find that the number of solutions of $ 37x + 23y + 16z =  2069613$ is given by $$ 157294695 + \sum_{i=1}^{32} \left \lfloor \frac{13i}{37} \right \rfloor + \sum_{i=1}^{16} \left \lfloor \frac{11i}{23} \right \rfloor + \sum_{i=1}^{6} \left \lfloor \frac{3i}{16} \right \rfloor. $$ Using the algorithm in \cite[Section 2.3]{Binner} (Basically reciprocity relation of \cite[Lemma 7]{Binner}), we easily get that 
 
\begin{equation}
 \begin{aligned}
 \sum_{i=1}^{32} \left \lfloor \frac{13i}{37} \right \rfloor &= 170, \\
 \sum_{i=1}^{16} \left \lfloor \frac{11i}{23} \right \rfloor &= 56, \\
\sum_{i=1}^{6} \left \lfloor \frac{3i}{16} \right \rfloor &= 1.
 \end{aligned}
\end{equation}

Thus, the number of solutions of $ 37x + 23y + 16z = 2069613$ is equal to $157294920$. Therefore, $k= 157294925$ is of category II, and we get that in this case $R_k(37,23,16) = \emptyset$. Thus, $c_{=k}(37,23,16) = s_{=k}(37,23,16) = 0$. However, in this case, $g_{=k}(37,23,16)$ and  $h_{=k}(37,23,16)$ are not defined.

\subsection{$R_k(a,b,c)$ for any natural numbers $a,b$ and $c$}

Next suppose $a, b$ and $c$ are any natural numbers with $\gcd(a,b,c) = 1$ (need not be pairwise coprime). We handle this case using the technique of reduction to an equation with pairwise coprime coefficients, described in \cite[Lemma 3]{Binner}. We restate that result here. For that, we recall the following notation.

\begin{itemize}
\item Let $g_1$, $g_2$, and $g_3$ denote $\gcd(b, c)$,  $\gcd(c, a)$, and $\gcd(a, b)$, respectively. Note that $\gcd(g_1, g_2) = \gcd(g_2, g_3) = \gcd(g_3, g_1) = 1$. 
\item Let $a_1$, $b_2$, and $c_3$ denote the modular inverses of $a$ with respect to the modulus $g_1$,  $b$  with respect to the modulus $g_2$, and $c$ with respect to the modulus $g_3$, respectively.
\item Let $n_1$, $n_2$, and $n_3$ denote the remainders upon dividing $na_1$ by $g_1$, $nb_2$ by $g_2$, and $nc_3$ by $g_3$, respectively.
\item Let $A =  \frac{a}{g_2 g_3}$,  $B =  \frac{b}{g_3 g_1}$, and $C =   \frac{c}{g_1 g_2}$. Note that $\gcd(A,B)=\gcd(B,C)=\gcd(A,C)=1$.
\item Let $N =  \frac{n - an_1 - bn_2 - cn_3}{g_1 g_2 g_3}$.
\end{itemize}

\begin{lemma}(B.(2020))
\label{Reduction}
With the notation above, the number of solutions of the equation $ax + by + cz = n$ in nonnegative integer triples $(x,y,z)$ is equal to the number of solutions of the equation $Ax + By + Cz = N$ in nonnegative integer triples $(x,y,z)$.
\end{lemma}

In particular, note that $ax + by + cz = n$ has $k$ solutions if and only if $Ax + By + Cz = N$ has $k$ solutions. 

\begin{itemize}
  \item Let $\alpha' = \frac{A+B+C}{2}$ and $\beta' = 2ABC$, and $$ M' = \left \lfloor \frac{(2\alpha'\beta'-1)^2 - 4\alpha'^2}{4\beta'} + \alpha' \right \rfloor + 1. $$
  \item  For any natural number $k$, define $$\gamma'_k = \left \lfloor \sqrt{\beta'(k+\alpha')+\alpha'^2} - \alpha' \right \rfloor$$ and $$\delta'_k = \left \lfloor \sqrt{\beta'(k-\alpha')+\alpha'^2} - \alpha' \right \rfloor. $$ 
  \end{itemize}
  
  Suppose $k \geq M'$ be of category II with respect to $A, B$ and $C$, then there is no value of $N$ such that $Ax + By + Cz = N$ has $k$ solutions. Thus, there is no value of $n$ such that $ax + by + cz = n$ has $k$ solutions. Thus, $R_k(a,b,c) = \emptyset$, and $$c_{=k}(a,b,c) = s_{=k}(a,b,c) =0.$$  In this case, $g_{=k}(a,b,c)$ and  $h_{=k}(a,b,c)$ are not defined.
  
  Suppose $k \geq M'$ be of category I with respect to $A, B$ and $C$, then $R_k(A,B,C) = \gamma'_k$. That is, $Ax + By + Cz = N$ has $k$ solutions if and only if $N = \gamma'_k$. That is $$ \frac{n - an_1 - bn_2 - cn_3}{g_1 g_2 g_3} = \gamma'_k. $$ Thus, $n$ lies in the set $$ \left\{g_1 g_2 g_3 \gamma'_k + ai_1 + bi_2 + ci_3: 0 \leq i_1 \leq g_1 - 1, 0 \leq i_2 \leq g_2 - 1, 0 \leq i_3 \leq g_3 - 1 \right\}. $$ Conversely, it is easy to see that if $n$ is a member of this set, then $N = \gamma'_k$, and thus $Ax+By+Cz=N$ has $k$ solutions, and therefore $ax+by+cz=n$ also has $k$ solutions. Thus, we get that $$ R_k(a,b,c) = \left\{g_1 g_2 g_3 \gamma'_k + ai_1 + bi_2 + ci_3: 0 \leq i_1 \leq g_1 - 1, 0 \leq i_2 \leq g_2 - 1, 0 \leq i_3 \leq g_3 - 1 \right\}. $$ Therefore, 
  
 $$ g_{=k}(a,b,c) = \max\{t: t \in R_k(a,b,c)\} = g_1 g_2 g_3 \gamma'_k + a(g_1-1) + b(g_2-1) + c(g_3-1),$$ 
 $$h_{=k}(a,b,c) = \min\{t: t \in R_k(a,b,c)\} = g_1 g_2 g_3 \gamma'_k, $$
$$c_{=k}(a,b,c) = |\{t: t \in R_k(a,b,c)\}| =  g_1 g_2 g_3. $$

Finally,

\begin{align*}
s_{=k}(a,b,c) &= \sum\{t: t \in R_k(a,b,c)\}. \\
&= \sum_{i_1=0}^{g_1-1} \sum_{i_2=0}^{g_2-1} \sum_{i_3=0}^{g_3-1} \left( g_1 g_2 g_3 \gamma'_k + ai_1 + bi_2 + ci_3 \right) \\
& = (g_1 g_2 g_3)^2 \gamma'_k + g_1 g_2 g_3\left(a(g'_1-1)+b(g'_2-1)+c(g'_3-1)\right). 
\end{align*}

 Let $a$, $b$ and $c$ be given positive integers such that $\gcd(a,b,c) = 1$. Let $g_1$, $g_2$, and $g_3$ denote $\gcd(b, c)$,  $\gcd(c, a)$, and $\gcd(a, b)$, respectively. Moreover, let $A =  \frac{a}{g_2 g_3}$,  $B =  \frac{b}{g_3 g_1}$, and $C =   \frac{c}{g_1 g_2}$. Then, we say that a number $k$ is of \emph{category I with respect to $a,b$ and $c$} if $k$ is of category I with respect to $A,B$ and $C$. Otherwise, we say that $k$ is of \emph{category II with respect to $a,b$ and $c$}. 

Finally suppose $a,b$ and $c$ be any natural numbers ($\gcd(a,b,c)$ need not be $1$). Let $\gcd(a,b,c) = g$. For $ax+by+cz=n$ to have $k \geq 1$ solutions, it is necessary that $g$ divides $n$. Let $a'' = \frac{a}{g}, b'' = \frac{b}{g}, c'' = \frac{c}{g}$, and $n'' = \frac{n}{g}.$ Note that $\gcd(a'',b'',c'')=1$.

\begin{itemize}
\item Let $g''_1$, $g''_2$, and $g''_3$ denote $\gcd(b'', c'')$,  $\gcd(c'', a'')$, and $\gcd(a'', b'')$, respectively. 
\item Let $A'' = \frac{a''}{g''_2g''_3}$, $B'' = \frac{b''}{g''_1g''_3}$ and $C'' = \frac{c''}{g''_1g''_2}$.
  \item Let $\alpha'' = \frac{A''+B''+C''}{2}$ and $\beta'' = 2A''B''C''$, and $$ M'' = \left \lfloor \frac{(2\alpha''\beta''-1)^2 - 4\alpha''^2}{4\beta''} + \alpha'' \right \rfloor + 1. $$
  \end{itemize}

Suppose $k \geq M''$. It is clear that $ax+by+cz=n$ has $k$ solutions if and only if $a''x+b''y+c''z=n''$ has $k$ solutions. That is, $n \in R_k(a,b,c)$ if and only if $n'' = \frac{n}{g} \in R_k(a'',b'',c'')$. Therefore, 
\begin{equation}
\label{Related}
 R_k(a,b,c) = \{gt: t \in R_k(a'',b'',c'')\}, 
 \end{equation}
  where $R_k(a'',b'',c'')$ can be calculated from the formulas described above (since $\gcd(a'',b'',c'')=1$). Note that if $k$ is of Category II with respect to $a'', b''$ and $c''$, then $R_k(a'',b'',c'') = \emptyset$, and therefore $R_k(a,b,c) = \emptyset$. Thus, $$ c_{=k}(a,b,c) = s_{=k}(a,b,c) = 0,$$ and the quantities $g_{=k}(a,b,c)$ and $h_{=k}(a,b,c)$ are not defined. If $k$ is of Category I with respect to $a'', b''$ and $c''$, then from \eqref{Related}, we have 
  $$ g_{=k}(a,b,c) = gg_{=k}(a'',b'',c''), $$
   $$h_{=k}(a,b,c) =gh_{=k}(a'',b'',c''), $$
   $$c_{=k}(a,b,c) =c_{=k}(a'',b'',c''), $$
   $$ s_{=k}(a,b,c) =gs_{=k}(a'',b'',c''), $$
   where $g_{=k}(a'',b'',c''), h_{=k}(a'',b'',c''), c_{=k}(a'',b'',c'')$, and $s_{=k}(a'',b'',c'')$ can be calculated from the formulas described above (since $\gcd(a'',b'',c'')=1$).

  \section{An application to a recent conjecture}

Next, we use the bounds in Theorem \ref{Bounds} to disprove a recent conjecture of He, Shiue and Venkat in \cite{Shiue}. First, we describe some notation required to state their conjecture. Let $a,b,c$ and $n$ be given positive integers with $\gcd(a,b,c)=1$. 

\begin{itemize}
\item $\hat{S_1}$ denotes the set of non-negative integer solutions $(x,y,z)$ of $ax+by+cz=n$ such that $x=0$.
\item $\hat{S_2}$ denotes the set of non-negative integer solutions $(x,y,z)$ of $ax+by+cz=n$ such that $y=0$.
\item $\hat{S_3}$ denotes the set of non-negative integer solutions $(x,y,z)$ of $ax+by+cz=n$ such that $z=0$.
\end{itemize}

\begin{conjecture}(He, Shiue and Venkat (2021))
For any solution $(\hat{x},\hat{y},\hat{z})$ of $ax+by+cz=n$, there exist $s_i \in \hat{S_i}$ such that $$ (\hat{x},\hat{y},\hat{z}) = s_1 - s_2 + s_3. $$
\end{conjecture}

They verified this conjecture for some basic examples. Moreover, they also gave the following interesting consequence of this conjecture, if it is true.

\vspace{.5cm}
\textbf{Consequence of the conjecture:} Let $N_i$ denote $|\hat{S_i}|$, and $\hat{N} = N_1 + N_2 + N_3$. Then $$ 0 \leq N(a,b,c;n) \leq 3 \binom{\hat{N}}{3}.  $$

We give a counterexample to this consequence of the conjecture, which in turn will disprove the conjecture. Choose $a=191$, $b=131$, $c=117$, and $n=67529$. Then, using the formula given in \cite[Corollary 17]{Binner}, $N_1 = 4$, $N_2 = 3$ and $N_3 = 3$. This method was also described independently in \cite{Ab}. Equivalently, we might calculate $N_1, N_2$ and $N_3$ using the formula given in \cite{AT}. Thus, $\hat{N} = 10$. Therefore, assuming the conjecture is true, the number of solutions of $$ 191x+131y+117z = 67529 $$ should be less than or equal to $3 \binom{10}{3} = 360$. However, by our bounds in Theorem \ref{Bounds}, we have $$ 565 \leq N(191,131,117;67529) \leq 1003, $$ giving the required contradiction. For further studies, it may be interesting to see if this conjecture or its consequence hold true for some special families of values of $a$, $b$ and $c$. 

\section{Acknowledgement}

I wish to thank the Maths Department at SFU for providing me various awards and fellowships which help me conduct my research.

\end{document}